\documentclass[11pt,reqno]{amsart}
\usepackage{amsthm,amsfonts,amssymb,amsmath,enumitem}
\numberwithin{equation}{section}
\usepackage{fullpage}

\usepackage{mathtools}
\mathtoolsset{showonlyrefs}

\newtheorem{theorem}{Theorem}[section]
\newtheorem{proposition}[theorem]{Proposition}

\theoremstyle{remark}
\newtheorem{remark}[theorem]{Remark}
\theoremstyle{definition}
\newtheorem{definition}[theorem]{Definition}
\newtheorem{problem}[theorem]{Problem}

\newtheorem{example}[theorem]{Example}

\newcommand{\cF}{\mathcal{F}}

\newcommand{\cO}{\mathcal{O}}

\newcommand{\N}{\mathbb{N}}
\newcommand{\R}{\mathbb{R}}
\newcommand{\C}{\mathbb{C}}

\renewcommand\a{\alpha}
\newcommand{\e}{\varepsilon}
\renewcommand{\epsilon}{\varepsilon}

\newcommand{\dd}{\,\mathrm{d}}
\newcommand{\DD}{\mathrm{d}}

\title{On nonlinear stabilization of linearly unstable maps}

\author{Thierry Gallay}
\address{
Institut Fourier, Universit\'e Grenoble Alpes, CNRS, 
F-38000 Grenoble, France}
\email{Thierry.Gallay@univ-grenoble-alpes.fr}
\author{Benjamin Texier}
\address{Institut de Math\'ematiques de Jussieu-Paris Rive Gauche,
Universit\'e Paris-Diderot, UMR CNRS 7586}
\email{benjamin.texier@imj-prg.fr}
\author{Kevin Zumbrun}
\address{Indiana University, Bloomington, IN 47405}
\email{kzumbrun@indiana.edu} 
\thanks{
Research of K.Z. was partially supported
under NSF grant no. DMS-0300487.
}
\begin{document}

\begin{abstract}
We examine the phenomenon of nonlinear stabilization, exhibiting a
variety of related examples and counterexamples. For
G\^ateaux differentiable maps, we discuss a mechanism of nonlinear
stabilization, in finite and infinite dimensions, which applies in
particular to hyperbolic partial differential equations, and, for
Fr\'echet differentiable maps with linearized operators that are
normal, we give a sharp criterion for nonlinear exponential
instability at the linear rate. These results highlight the
fundamental open question whether Fr\'echet differentiability is
sufficient for linear exponential instability to imply nonlinear
exponential instability, at possibly slower rate.
\end{abstract}
\date{\today}
\maketitle

\section{Introduction}\label{intro}

Since the pioneering work of Lyapunov, it has been classical to deduce
stability properties of equilibria of dynamical systems from spectral
information on the linearized operator. In his memoir~\cite{Ly},
Lyapunov considers general systems of ordinary differential equations
with analytic coefficients, in finite dimensions, and establishes
several fundamental results on stability or instability of equilibria
or time-periodic solutions, using spectral information on appropriate
linearized systems. In particular, an equilibrium solution of an
autonomous system is shown to be asymptotically stable if all
eigenvalues of the associated linearized operator have strictly
negative real parts, and to be unstable if at least one eigenvalue has
a strictly positive real part. Such results were subsequently
generalized to infinite-dimensional systems and less regular
nonlinearities, under suitable spectral assumptions. The reader 
is referred to the monograph of Daleckii and Krein~\cite{DK} 
for a nice account of these early studies. 

In the literature that follows Lyapunov's work, there is a striking
asymmetry between the generalizations of the stability and the
instability theorem, respectively. On the one hand, assuming the
spectral mapping theorem and using a suitable norm, it is relatively
straightforward to show that, for any autonomous dynamical system, an
equilibrium is asymptotically stable if the vector field of the system
is Fr\'echet differentiable at that point and if the spectrum of the
derivative is entirely contained in the open left half-plane. In
contrast, the Lyapunov instability theorem has no counterpart so far
at this level of generality: many sufficient conditions for
instability are known, but they all require either the existence of a
spectral gap, or a somewhat restrictive assumption on the
nonlinearity.

In fact, as of now, we are not aware of any example of nonlinear
stabilization for a linearly unstable Fr\'echet differentiable
dynamical system, nor of any result that would prevent such a
phenomenon to occur under minimal and natural assumptions. The modest
goal of this paper is to discuss the existing results in this
direction, and to give a few generalizations. We also present
examples that illustrate various aspects of this fundamental open
question. 

\subsection{Formulation of the problem} \label{sec:setting} 

To avoid technicalities related to unbounded linear operators, we find
it convenient to formulate the problem in terms of discrete time
systems, namely difference equations, rather than differential
equations as in Lyapunov's work. We thus consider a discrete
evolutionary system of the form
\begin{equation}\label{de} 
 u_{n+1} \,=\, \cF(u_n)\,, \qquad n \ge 0\,, 
\end{equation} 
in a Banach space $B$, where $\cF : B\to B$ is a nonlinear map.  We
always assume that $\cF(0)=0$, so that the origin $u = 0$ is an
equilibrium point. This general framework includes in particular
continuous time-evolutionary systems, such as partial differential
equations with time-independent or time-periodic coefficients, in
which case $\cF$ is defined as the time-$T$ solution map, see
Section~\ref{sec:comments}. As discussed in point (5) of Section
\ref{sec:remarks} below, partial differential operators with smooth
coefficients do not necessarily generate smooth solutions maps. We
suppose nevertheless that there is associated with $\cF$ a linearized
map $L$, which is typically a Fr\'echet or G\^ateaux derivative of
$\cF$ at the origin.

The main question is:

\begin{problem}\label{mainprob}
{\em Assuming that $L$ has spectral radius $r(L)>1$,
corresponding to linear exponential instability, under what general 
conditions may we deduce also nonlinear instability?}
\end{problem}

By nonlinear instability, we mean instability in the sense of Lyapunov
of the equilibrium $u=0$ for the evolutionary system \eqref{de}. A
related question is whether exponential instability occurs in the
situation described by Problem~\ref{mainprob}, and, if so, whether it
occurs at the linear rate $\rho = r(L)$. Here are the precise
definitions:

\begin{definition}\label{def:nins} 
The equilibrium $u = 0$ is
\begin{itemize}[leftmargin=15pt,topsep=3pt,itemsep=3pt]

\item {\em unstable (in the sense of Lyapunov)} if there exists 
$\epsilon > 0$ such that, for any $\delta > 0$, one can find a 
sequence $(u_n)$ solution of \eqref{de} such that $0 < |u_0|_B \le 
\delta$ and $|u_n|_B \ge \epsilon$ for some $n \in \N$. 

\item {\em exponentially unstable at rate}
$\rho > 1$ if there exists $\epsilon > 0$ and $C > 0$ such that, 
for any $\delta > 0$, one can find a sequence  $(u_n)$ solution of 
\eqref{de} satisfying $0 < |u_0|_B \le \delta$ and $|u_n|_B
\ge C \rho^n |u_0|_B$ for all $n \in \N$ such that 
$\max(|u_0|_B,\dots ,|u_n|_B) \le \epsilon$. 
\end{itemize}

\end{definition}

It is obvious that exponential instability at any rate $\rho > 1$ 
implies instability in the sense of Lyapunov, with the same value 
of $\epsilon$. 

\subsection{Classical results} \label{sec:classical} 

Our understanding of Problem~\ref{mainprob} is grounded in  
the following classical results:

\smallskip
(1) The {\it Lyapunov instability theorem} \cite{Ly}, which states 
that if $B$ is finite-dimensional and if the linear operator 
$L : B \to B$ with $r(L) > 1$ approximates $\cF$ near the origin
in the sense that, for some $a > 0$ and $b > 0$ small enough,  
\begin{equation}\label{remest1}
 |\cF(u) - Lu|_B \,\le\,  b|u|_B \qquad\hbox{whenever}\quad 
 |u|_B \le a\,,   
\end{equation}
then the equilibrium $u=0$ is exponentially unstable. Remark that 
condition~\eqref{remest1} is always fulfilled if $\cF$ 
is Fr\'echet differentiable at the origin and $L=d\cF|_{u=0}$. 
Even in that particular case, instability may not occur at 
the linear rate, as is shown by Proposition~\ref{ex:scal}. 

\smallskip
(2) The extension of Lyapunov's instability theorem to the 
infinite-dimensional case under the assumption of a {\it spectral
gap}, see e.g. \cite[Theorem~VII.2.2]{DK}. Precisely, let 
$L$ be a bounded linear operator on the Banach space $B$ 
with spectral radius $r(L) > 1$, and assume that the 
spectrum $\sigma(L)$ does not intersect the unit circle 
$\{\lambda \in \C\,|\,|\lambda| = 1\}$. Then, if 
the map $\cF : B \to B$ satisfies \eqref{remest1} for 
some $a > 0$ and $b > 0$ small enough, the equilibrium $u=0$ is 
exponentially unstable. More generally, the same conclusion
is obtained if one assumes that the linear operator $L$ is 
$\rho$-pseudohyperbolic \cite{Ru} instead of hyperbolic, 
namely if the spectrum $\sigma(L)$ does not intersect the 
circle of radius $\rho$, for some $\rho \in [1,r(L))$. It 
should be emphasized that the smallness assumption
on the parameter $b$ in \eqref{remest1} depends on the 
linear operator $L$, in such a way that $b \to 0$ when 
the spectral gap shrinks to zero. 

\smallskip
(3) The {\it Rutman-Daleckii theorem} \cite[Theorem~VII.2.3]{DK}, 
and its more general version due to Henry \cite[Theorem 5.1.5]{He}, 
which states that if $L : B \to B$ is a bounded linear operator
which approximates the map $\cF$ at a superlinear rate at the origin 
in the sense that 
\begin{equation}\label{remest2}
  |\cF(u) -Lu |_B \,\le\, b|u|_B^{1 + p} \qquad\hbox{whenever}\quad 
  |u|_B\le a\,,
\end{equation}
for some $a>0$, $b>0$, and $p > 0$, then $r(L) > 1$ implies nonlinear
exponential instability of the origin, at the linear rate $\rho =
r(L)$. Condition \eqref{remest2} is of course stronger than
\eqref{remest1}, and implies in particular that $\cF$ is Fr\'echet
differentiable at the origin with $L=d\cF|_{u=0}$. It holds 
typically if the map $\cF$ has the regularity $C^{1,p}$ near the
origin, for some $p \in (0,1]$, namely if the Fr\'echet derivative $u
\mapsto d\cF(u)$ is H\"older continuous with exponent $p$. Note that,
in the Rutman-Daleckii theorem, the spectral instability condition
$r(L) > 1$ is the only assumption made on the linear operator $L$.

\smallskip
If the map $\cF$ in \eqref{de} is Fr\'echet differentiable at the
origin, the results above show that linear exponential instability
implies nonlinear instability if the linearized operator
$L=d\cF|_{u=0}$ either has a spectral gap, or approximates
the full map $\cF$ at a superlinear rate. Both conditions are
sufficient, but neither one is known to be necessary. 

\subsection{Remarks and examples} \label{sec:remarks}

We do not give a definite answer to Problem \ref{mainprob}, but only 
make a few remarks which, we hope, shed some light on what the solution 
may be. 

\smallskip
(4)  We first recall that even a very weak nonlinearity can stabilize
a map that is linearly unstable, but not at exponential rate. 
This phenomenon already happens in finite dimensions, as is 
demonstrated by the following simple example
\begin{equation}\label{2Dsys}
  \cF(v,w) \,=\, L \begin{pmatrix} v \\ w \end{pmatrix} - 
  \begin{pmatrix} v^3 \\ w^3 \end{pmatrix}, \qquad
  L = \begin{pmatrix} 1 & 1 \\ 0 & 1 \end{pmatrix},
\end{equation}
for which the equilibrium $(v,w) = (0,0) \in B = \R^2$ is linearly
unstable, in the sense that $L^n$ is unbounded as $n \to \infty$, and
yet nonlinearly asymptotically stable, see Section~\ref{sec:app}.
In that case we have of course $r(L) = 1$.

In Section~\ref{sec:Jordan} below, we construct in the same spirit an
infinite-dimensional map $\cF$ for which the origin $u=0$ is stable,
although the linearized operator $L=d\cF|_{u=0}$ has the property that
$L^n$ grows nearly exponentially as $n \to \infty$. This shows that
the linear exponential instability assumption $r(L) > 1$ is essential
in Problem~\ref{mainprob}.

\smallskip
(5)  Our next contribution is an intriguing example which indicates
that, in the absence of a spectral gap, stabilization may be 
possible if the nonlinearity does not satisfy a superlinearity 
condition such as \eqref{remest2}. 

\begin{example}\label{Gallex}
Let $\chi : \R \to [0,2]$ be a smooth function such that ${\rm
supp}(\chi) \subset [-1,1]$ and $\chi(0) = 2$.  Let also $h :
[0,\infty) \to [0,\infty)$ be an increasing continuous function such
that $h(0) = 0$.  In the Hilbert space $B = L^2(\R)$, we consider
the map $\cF : B \to B$ defined by
\begin{equation}\label{def:Gal}
  \bigl(\cF(u)\bigr)(x) \,=\, \chi(x)\,u\bigl(x - h(|u|_B)\bigr)\,, 
  \qquad x \in \R\,,
\end{equation} 
for all $u \in B$. Then $\cF(0) = 0$, and $\cF$ is differentiable at
the origin in the sense of G\^ateaux, but not in the sense of
Fr\'echet, with derivative $L$ given by
\[
  (Lu)(x) \,=\, \chi(x)\,u(x)\,, \qquad x \in \R\,.
\]
The operator $L$ is clearly self-adjoint in $B$ with spectrum
$\sigma(L) = [0,2]$, so that the origin $u = 0$ is linearly
exponentially unstable. However, as is shown in Section~\ref{sec:pde}
below, the origin is nonlinearly stable if $h(s)$ converges
sufficiently slowly to zero as $s \to 0$, for instance if $h(s) 
\ge C|\ln s|^{-1}$ for some sufficiently large $C > 0$.
\end{example}

Example~\ref{Gallex} shows in particular that nonlinear stabilization
is possible if the linearization of the map $\cF$ at the origin is
understood in a weaker sense than the usual Fr\'echet derivative. We
elaborate on that question in Section~\ref{sec:pde}, where we first
prove all assertions above and then discuss a few related examples.
One of them corresponds to a seemingly minor modification 
of the map \eqref{def:Gal}, namely
\begin{equation}\label{def:Gal2}
 \bigl(\cF(u)\bigr)(x) \,=\, \chi(x)\,u\bigl(2x - h(|u|_B)\bigr)\,, 
  \qquad x \in \R\,,
\end{equation}
which has nevertheless a very different behavior. Unlike in
Example~\ref{Gallex}, nonlinear stabilization is now possible even for
a very smooth nonlinearity, such as $h(s) = s^2$. More remarkably, we
also provide an example of a simple hyperbolic partial differential
equation, for which the time-one solution map behaves qualitatively as
in \eqref{def:Gal2}, so that nonlinear stabilization occurs. The
choice of a hyperbolic equation here is not accidental: in such
systems, the flow is typically not Fr\'echet differentiable at
equilibria, so that weaker notions of linear tangent maps have to be
introduced. In contrast, parabolic equations with smooth
nonlinearities typically generate Fr\'echet differentiable flows, in
which case the Rutman-Daleckii theorem applies and shows, in view of
\eqref{remest2}, that nonlinear stabilization is impossible if the
flow is sufficiently smooth, for instance if the differential $u
\mapsto d\cF(u)$ is H\"older continuous with exponent $p > 0$.
Finally, we also give in Section~\ref{sec:pde} a finite-dimensional
example of a G\^ateaux differentiable map for which nonlinear
stabilization occurs.

\smallskip
These examples show that, even in finite dimensions, it is necessary in
Problem~\ref{mainprob} to understand the linearized operator $L$ in
the usual Fr\'echet sense.

\smallskip
(6) In Section~\ref{sec:normal}, we study in some detail the important
particular case where $B$ is an infinite-dimensional Hilbert space,
the map $\cF : B \to B$ is Fr\'echet differentiable at the origin, and
the linearized operator $L=d\cF|_{u=0}$ is selfadjoint or normal.  In
that situation, even in the absence of a spectral gap, it is possible
to separate the unstable part of the spectrum using spectral
projections. We prove that linear exponential instability implies
nonlinear instability provided that, for some $a > 0$,
\begin{equation}\label{remest3}
  |\cF(u) - Lu|_B \,\le\, \alpha(|u|_B)|u|_B\,, 
  \qquad\hbox{whenever}\quad |u|_B\le a\,,
\end{equation}
where $\alpha : [0,a] \to [0,+\infty)$ is a nondecreasing function 
satisfying the integrability condition
\begin{equation}\label{intalpha}
  \int_0^a \frac{\alpha(s)}{s}\dd s \,<\, \infty\,. 
\end{equation}
This condition is obviously satisfied if $\alpha(s) = b s^p$ for
some $p > 0$, hence our result sharpens the Rutman-Daleckii theorem in
the normal case by replacing \eqref{remest2} with the weaker
assumption \eqref{remest3}-\eqref{intalpha}.

It is interesting to observe that the function $\a(s) = |\ln s|^{-\gamma}$ 
satisfies \eqref{intalpha} for any $\gamma > 1$. As for the 
limiting case $\gamma = 1$, Example~\ref{Gallex} shows on the 
contrary that nonlinear stabilization can occur for a map $\cF$ 
such that
\begin{equation} \label{XB} 
  |\cF(u) - L u|_B \,\le\, C |\ln |u|_B|^{-1} |u|_X\,, \qquad 
  \hbox{for all }u \in X \subset B\,,
\end{equation}
where $X = H^1(\R)$, $B = L^2(\R)$, and $L : B \to B$ is a selfadjoint 
operator. This is an indication that condition~\eqref{intalpha} may be 
close to optimal in the normal case -- an indication but not a proof, 
since the nonlinear map in Example~\ref{Gallex} is not Fr\'echet 
differentiable, as is reflected in the fact that $X \neq B$ 
in \eqref{XB}.

\subsection{Conclusion} \label{sec:conclusion} 

We are now in position to give a more precise formulation of 
Problem~\ref{mainprob}:

\begin{problem}\label{mainprob2}
{\em Assuming Fr\'echet differentiability of the map $\cF$ at the 
origin and exponential instability of the linearized operator
$L=d\cF|_{u=0}$, corresponding to $r(L) > 1$, what are sharp conditions 
on the remainder $|\cF(u) -Lu|_B$ for {\rm (a)} nonlinear instability,  
{\rm (b)} nonlinear exponential instability, {\rm (c)}
nonlinear exponential instability at linear rate $\rho = r(L)$?}
\end{problem}

In the normal case, the results of Section~\ref{sec:normal} give 
a satisfactory answer to (c), and based on Example~\ref{Gallex} 
we conjecture that conditions \eqref{remest3}, \eqref{intalpha} 
are also sharp for (b). In the general case, these questions 
remain essentially open, although one may conjecture that 
nonlinear stabilization, if it occurs at all, requires both
nonseparability of the spectrum and a very slow vanishing 
rate of the nonlinearity.

\subsection{Further comments and applications} 
\label{sec:comments} 

We try here to make a connexion between the abstract questions
discussed in this paper and some concrete problems studied in the
literature, especially in fluid mechanics where stability issues are
of great theoretical and practical importance.  Consider equations in
the general form
\begin{equation}\label{pde} 
  \frac{\DD u}{\DD t} \,=\, Au + f(u)\,, \qquad t \ge 0\,, 
\end{equation} 
where the linear operator $A : D(A) \to B$ is the generator of a
strongly continuous semigroup in a Banach space $B$, and the 
nonlinearity $f : B \to B$ is locally Lipschitz and satisfies 
\begin{equation} \label{est:f0} 
  |f(u)|_B \,=\, o(|u|_B), \qquad \hbox{as }  |u|_B \to 0\,.
\end{equation} 
In this case, any solution $u_n = u(n)$ of \eqref{pde} evaluated at
integer times satisfies the recursion relation \eqref{de}, where $\cF$
is the time-1 solution map, and at the linear level one has the
relation $L=d\cF|_{u=0} = \exp(A)$. Most of the results presented
above for the difference equation \eqref{de} have their counterpart
for the differential equation \eqref{pde}-\eqref{est:f0}, in
particular the Lyapunov instability theorem under the assumption of a
spectral gap and the Rutman-Daleckii theorem \cite{SS}.

Many time-evolutionary partial differential equations, while admitting
the general form \eqref{pde} (with the possible addition of constraint
equations and boundary conditions), have nonlinearities $f$ which are
not Lipschitz. Often $f(u)$ involves spatial derivatives of the
unknown $u,$ and satisfies an estimate of the form
\begin{equation} \label{est:f}
  |f(u)|_B \,\lesssim\, |u|_X |u|_B\,, \qquad \hbox{for small 
  enough }|u|_X\,, \hbox{ with } X \hookrightarrow B\,.
\end{equation}
The loss of regularity in estimate \eqref{est:f} can sometimes be
compensated for by regularizing estimates for the semigroup. For the
Navier-Stokes equations, these are provided by analytic semigroup
theory; for Schr\"odinger and related dispersive operators, by
Strichartz-type estimates. In the absence of regularizing estimates,
particular features of the system can be exploited, such as the
cancellation $(f(u),u)_B = 0,$ with $B = L^2$, which holds for perfect
incompressible fluids governed by the Euler equations and delimited
by impermeable boundaries. In any case, it is necessary to deal with
several function spaces. In particular, the assumptions of stability
and instability theorems become complicated and case-dependent when
formulated at the level of the differential equation
\eqref{pde}-\eqref{est:f}. For examples of such statements, we refer
to \cite{SS}.

A great advantage of working with the ``integrated'' formulation
\eqref{de} is that all results can be stated in a single Banach space,
which typically consists of sufficiently regular functions so that the
Cauchy problem for \eqref{pde} is locally well-posed in this
space. This point of view is satisfying in terms of simplicity and
generality, and avoids any confusion between the assumptions of the
stability/instability theorems and what is needed to solve the Cauchy
problem. However, in most applications, spectral information on the
linearized problem is easier to obtain at the level of the generator
$A$, which is typically an explicit differential operator, whereas the
semigroup $\exp(A)$ has no simple representation. In addition, it
often happens that the spectrum of $A$ is more conveniently studied in
a low regularity space such as $L^2$, while the Cauchy problem is only
known to be well-posed in a smaller space, such as the Sobolev space
$H^s$ for sufficiently large $s > 0$. Finally, the notions themselves
of stability or instability may be sensitive to the choice of the
function space (see e.g.~\cite{Li1}), and it is not clear that using a
framework where the Cauchy problem is well-posed gives the most
appropriate definition of nonlinear instability. In fact, a very
conclusive notion of instability is obtained if small initial
perturbations in a strong norm (e.g. in $H^s$ for large $s > 0$) are
shown to evolve into large discrepancies measured in a weak norm
(e.g. in $L^2$).  Strong instabilities in this sense have been
established by Grenier \cite{Gr} for the 2D Euler equation, and the
same approach was subsequently used to establish transverse
instabilities of travelling waves in dispersive Hamiltonian models
\cite{RT1,RT2}.  In sum, we point out that the instability theorems
formulated for system \eqref{de} in a single Banach space, although
potentially applicable to a variety of situations, do not subsume the
numerous results obtained for particular PDEs, especially in fluid
mechanics \cite{FSV,Yu,BGS,VF,Li2,FPS} and for related models
\cite{GS,LS,FPV}.

We add a few comments concerning the scope of the results presented in
this paper. As for the linear part of the system, emphasis is put on
the situation where no spectral gap exists, so that no version of the
Lyapunov instability theorem can be invoked. This point of view is
quite reasonable, because in applications the linearized operator
often has continuous spectrum without any gap, especially in
nonparabolic equations or for systems on unbounded spatial domains. In
such situations, the Rutman-Daleckii theorem is applicable if the
solution map is sufficiently smooth, and this is why our results
concentrate on systems where the solution map is merely $C^1$, or is
even less regular so that the linearization at the origin has to be
understood in a weak sense, for instance as a G\^ateaux derivative.
Such a general point of view is not of pure academic interest: there
are many examples of partial differential equations for which the
solution map is not even of class $C^1$, no matter how smooth the
nonlinearities. Nonlinear hyperbolic systems typically belong to this
category \cite{Ka2}, whereas parabolic equations usually generate 
smooth solution maps if the nonlinearities are smooth. 

We conclude this introduction with two remarks on related questions.
First, we recall that a possible approach to prove nonlinear
instability in a system such as \eqref{de} or \eqref{pde} is to
construct an unstable invariant set, which contains negative
trajectories of the system that converge to the equilibrium as $n \to
-\infty$ or $t \to -\infty$. As is well known, if the equilibrium is
hyperbolic, the unstable set is a manifold that is as smooth as
allowed by the nonlinearity \cite{Ru}. Interestingly enough, if $B$ is
a Hilbert space and if the linearized operator is normal, it is
possible to construct an unstable invariant set even in the absence of
spectral gap, see \cite[Theorem~7.4]{EZ}, provided the nonlinearity is
sufficiently smooth. Such a result strengthens the conclusion of the
Rutman-Daleckii theorem in the normal case. From another perspective,
we would like to mention that Problem~\ref{mainprob} is strongly
related to a different question that can be formulated for the
differential equation \eqref{pde}: {\it assuming that the linearized
  equation $u' = Au$ is ill-posed, under what conditions can we deduce
  that the nonlinear system \eqref{pde} is also ill-posed?} This
question does not make sense for the difference equation \eqref{de},
and will therefore not be discussed here. We refer the interested
reader to \cite{BS,GT} for the analysis of a few examples.


\section{Stabilization of near-exponentially unstable 
linear maps} \label{sec:Jordan}

The classical example \eqref{2Dsys} of stabilization of a neutrally
unstable equilibrium can be extended in the infinite-dimensional case
to a much more dramatic example of stabilization of a near-exponentially 
unstable linear map. This shows that, for nonnormal operators in
infinite dimensions, any linear instability that is less than
exponential, in the sense that $r(L) = 1$, is susceptible of nonlinear
stabilization by a map $\cF$ satisfying estimate \eqref{remest2} for
some arbitrarily large $p > 0$. In our example, we have the lower
bound
\[ 
  |L^n|_{B\to B} ~\ge\, \prod_{1 \le k \le n} m_k\,, \qquad \hbox{for all }
  n \ge 1\,, 
\]
where $(m_k)$ is a nonincreasing sequence of real numbers converging to
$1$ arbitrarily slowly as $k \to \infty$. An appropriate choice of the
sequence $(m_k)$ thus leads to linear instability at any
sub-exponential rate.
  
Consider the space $B = \ell_2(\N)$ of square-integrable real
sequences $u = (u^0, u^1, u^2, \dots)$, and let $S : B \to B$ 
be the right shift defined by $Su = (0, u^0, u^1, \dots)$ for 
all $u \in B$. Given a nonincreasing sequence $(m_n)$ of real 
numbers such that $m_1 \le 2$ and $m_n \to 1$ as $n \to \infty$, 
we denote by $M : B \to B$ the associated multiplication 
operator: $M u = (m_0 u^0, m_1 u^1,\dots)$. We consider the map 
$\cF: B \to B$ defined by
\begin{equation}\label{FFdef}
  \cF(u) \,=\, (1 - |u|^p_B) M S u\,, \qquad u \in B\,,
\end{equation}
where $p > 0$. It is clear that $\cF$ is Fr\'echet differentiable 
at the origin, with linear tangent map $L = M S$. An easy calculation 
shows that, for any $n \ge 1$, 
\begin{equation}\label{MSn}
  |(MS)^n|_{B\to B} \,=\, |(MS)^n \big(1, 0, 0, \dots\big)|_B 
  ~=\, \prod_{1 \le k \le n} m_k\,,
\end{equation}
and this implies that the spectral radius of $MS$ is equal to 
one, because $m_n \to 1$ as $n \to \infty$. 

\begin{proposition} \label{prop:nil}
For any $p > 0$ and any sequence $(m_n)$ as above, the origin $u = 0$ 
is a stable fixed point of system \eqref{de} with $\cF$ given by 
$\eqref{FFdef}$. In addition, for the particular choice
\begin{equation}\label{mnspecial}
  m_n \,=\,  1 + \frac{1}{\ln (n+2)}\,,
\end{equation}
there exists $C > 0$ such that
\begin{equation}\label{low:lin}
  |(MS)^n|_{B\to B} \,=\, |(MS)^n \big(1, 0, 0, \dots\big)|_B 
  \,\ge\, C e^{n/(2 \ln n)}, \qquad \hbox{for all }n \ge 2\,,
\end{equation}
indicating near-exponential linear instability.
\end{proposition}

\begin{proof} We first remark that, if $|u|_B \le 1$, then 
\begin{equation}\label{bd:F} 
  |\cF(u)|_B \,\le\, 2 |u|_B\,,
\end{equation}
because $m_n \le m_1 \le 2$ for all $n \ge 1$. Let $(u_n)$ be the solution
of \eqref{de} with initial data $u_0 \in B$. Then $u_n$ has the form 
\[ 
  u_n \,=\, \Big(\underbrace{0, \dots, 0}_{\mbox{\footnotesize{$n$ terms}}}, \star\,, 
  \dots\Big)\,,
\]
due to the repeated action of the right shift. In particular, as long as 
$|u_n|_B \le 1$, there holds
\begin{equation}\label{bd:un}
  |u_{n+1}|_B \,\le\, m_{n+1} (1 - |u_n|^p_B) |u_n|_B\,, \qquad 
  \hbox{hence}\quad 
  |u_n|_B \,\le\,|u_0|_B \prod_{1 \le k \le n} m_k\,.
\end{equation}

Now, given any $\e \in (0,1)$, there exists $N = N(\e) \in \N^*$ such 
that  
\begin{equation}\label{Ne}
   m_N \Big(1 - \frac{\e^p}{2^p}\Big) \,<\, 1\,.
\end{equation}
Assume that $|u_0|_B < \delta$, where $\delta > 0$ is small enough so that 
\begin{equation} \label{de-e}
  \delta \prod_{1 \le n \le N} m_n \,<\, \e/2\,.
\end{equation}
Then, by the second inequality in \eqref{bd:un}, for all $n \in 
\{1,\dots,N\}$ there holds $|u_n|_B < \e/2$. In fact, we shall 
show that $|u_n|_B < \e$ for all $n \in \N$, which proves that 
the origin is a stable equilibrium. 

Indeed, if this is not the case, there exists a smallest integer 
$N_0 > N$ such that $|u_{N_0}|_B \ge \e$. Let $n = N_0-1 \ge N$, 
so that $|u_n|_B < \e$. If $|u_n|_B < \e/2$, estimate \eqref{bd:F} 
implies that 
\[
  |u_{N_0}|_B \,=\, |u_{n+1}|_B \,=\, |\cF(u_n)|_B 
  \,\le\, 2 |u_n|_B \,<\, \e\,,
\] 
which is a contradiction. On the other hand, if $\e/2 \le |u_n|_B < \e$, 
it follows from \eqref{Ne} that
\[
  m_{n+1}(1 - |u_{n}|_B^p) \,\le\, m_{N}(1 - \e^p/2^p) \,<\, 1\,,
\]
because $m_{n+1} = m_{N_0} \le m_{N}$ by construction. Using the first 
inequality in \eqref{bd:un}, we deduce that $|u_{N_0}|_B = |u_{n+1}|_B
< |u_n|_B < \e$, which is again a contradiction. Thus there 
exists no such $N_0$, and nonlinear stability is established.
  
\smallskip
We next derive a lower bound on the linear growth rate, in the 
particular case of the sequence $(m_n)$ given by \eqref{mnspecial}. 
In view of \eqref{MSn} we have
\[ 
  \ln |(MS)^n (1,0,0,\dots)|_B ~=\, \sum_{1 \le k \le n} 
  \ln(1 + \a_k)\,, \qquad \hbox{where}\quad \a_k \,=\, \frac{1}{\ln (k+2)}\,.
\]
But $\ln(1 + \a_k) \ge \a_k/2$ because $\a_k \in [0,2]$, hence
\[
  \ln |(MS)^n (1,0,0,\dots)|_B \,\ge\, \frac{1}{2} \sum_{1 \le k \le n} \a_k\,.
\]
Since $x \mapsto \ln(x+2)^{-1}$ is a decreasing function of $x \ge 0$, 
we deduce
\[
  \ln |(MS)^n (1,0,0,\dots)|_B \,\ge\, \frac{1}{2} \int_1^{n+1} 
  \frac{\dd x}{\ln(x+2)} \,\ge\, \frac12 \frac{x+2}{\ln(x+2)}
  \Big|_{x=1}^{x=n+1} \,=\, \frac{1}{2} \frac{n+3}{\ln(n+3)} - 
  \frac{3}{2 \ln 3}\,,
\]
and estimate \eqref{low:lin} easily follows. 
\end{proof}

In our example the linear operator $L = MS$ is not normal, because the
adjoint operator $L^* = S^*M$ (where $S^*$ is the left shift) does not
commute with $L$. In fact $L$ is a compact perturbation of the right
shift $S$, which in turn corresponds to an infinite-dimensional Jordan
block.  Remark that the phenomenon described in Proposition
\ref{prop:nil} cannot occur in the normal case, because spectral
stability of a normal operator $L$ implies that $|L|_{B\to B} = r(L)
\le 1$.  From this point of view, the example above is rather related
to the stabilization of unstable pseudospectra \cite{TE}. Note finally 
that Proposition~\ref{prop:nil} holds for arbitrarily large values 
of $p > 0$, and this is in sharp contrast with the situation described 
in the Rutman-Daleckii theorem, where the superlinearity condition 
\eqref{remest2} with any $p > 0$ is sufficient to prevent nonlinear 
stabilization. 


\section{Examples of nonlinear stabilization}\label{sec:pde}

We discuss here Example~\ref{Gallex} and several variants illustrating
the possibility of nonlinear stabilization for linearly exponentially
unstable equilibria. In all these examples, however, the map $\cF$ is
not Fr\'echet differentiable and the linearization has to be
understood in a weaker sense, typically as a G\^ateaux derivative.

\subsection{The main example} In $B = L^2(\R)$, we consider the map
$\cF : B \to B$ defined by \eqref{def:Gal}. We recall that $\chi : \R
\to [0,2]$ is a smooth function such that ${\rm supp}(\chi) \subset
[-1,1]$ and $\chi(0) = 2$, and that $h : [0,\infty) \to [0,\infty)$ is
an increasing continuous function satisfying $h(0) = 0$. 

\begin{proposition}\label{Gallprop}
Given any $u_0 \in B$, the solution of \eqref{de} satisfies 
$|u_n|_B \to 0$ as $n \to \infty$. In addition, for any $\delta > 0$
and any $u_0 \in B$ with $|u_0|_B \leq \delta$, the sequence $(u_n)$ 
issued from $u_0$ satisfies
\begin{equation}\label{Mdelest}
  \max_{n \in \N}|u_n|_B \,\le\, \delta\,2^{\frac{2}{h(\delta)}+1}\,,
\end{equation}
which proves nonlinear stability if $h(s)$ converges sufficiently slowly
to zero as $s \to 0$. 
\end{proposition}

\begin{proof}
Using the specific form of the map $\cF$ defined in \eqref{def:Gal}, 
we establish by induction on $k$ the following representation formula 
for the solution of \eqref{de}:
\begin{equation}\label{Gallun}
  u_n(x) \,=\, \biggl(\,\prod_{j=n-k+1}^n \chi(x-S_j^n)\biggr)\,
  u_{n-k}(x - S_{n-k}^n)\,, \qquad x \in \R\,, \quad n \ge k \geq 1\,,
\end{equation}
where the spatial shifts $S_j^n$ are defined by 
\[
   S_j^n \,=\, \sum_{\ell=j}^{n-1} h(|u_\ell|_B) \,\ge\, 0\,,
   \qquad 0 \le j \le n\,,
\]
with the convention that $S_j^n = 0$ if $j = n$. 

Given $\epsilon > 0$, let $I_\epsilon = \{n \ge 1\,|\,|u_n|_B \ge 
\epsilon\} \subset \N$. We claim that $I_\epsilon$ is a finite set, 
with cardinality
\begin{equation}\label{Nepsbdd}
  N_\epsilon \,=\, \mathrm{card}(I_\epsilon) \,\le\, \frac{2}{h(\epsilon)}+1\,. 
\end{equation}
This implies in particular that $|u_n|_B \to 0$ as $n \to \infty$. To
prove \eqref{Nepsbdd}, we first observe that, if $n \in I_\epsilon$,
then $S_1^n \le 2$. Indeed, if $S_1^n > 2$, then $\chi(x)\chi(x-S_1^n)
= 0$ for all $x \in \R$ by the support property of $\chi$, and
the relation \eqref{Gallun} with $k = n$ shows that $u_n \equiv 0$, 
so that $n \notin I_\epsilon$. Using the monotonicity of the function $h$, 
we deduce that, for any $n \in I_\epsilon$,
\[
  2 \,\ge\, S_1^n \,=\, \sum_{j=1}^{n-1} h(|u_j|_B) \,\ge\, \mathrm{card}
  \bigl(I_\epsilon \cap \{1,\dots,n-1\}\bigr) h(\epsilon)\,,
\]
and \eqref{Nepsbdd} easily follows. 

We next turn to a proof of \eqref{Mdelest}. Given any $\delta > 0$ and 
any $u_0 \in B$ with $|u_0|_B \le \delta$, we take $N \in \N$ such that 
\[
  |u_N|_B \,=\, \max_{n \in \N}|u_n|_B\,.
\]
Such an $N$ exists since the sequence $(|u_n|_B)_n$ converges to
zero. If $|u_N|_B \leq \delta$, then \eqref{Mdelest} is proved. 
Otherwise, in the backward sequence of real numbers $|u_N|_B$, 
$|u_{N-1}|_B, \dots, |u_0|_B$ we find the first term or terms to be 
greater than $\delta$, but not all, since $|u_0|_B \leq \delta$. 
That is, there exists an integer $k \le N$ such that $|u_{N-k}|_B \le 
\delta$ and $|u_j|_B > \delta$ for $j \in [N{-}k{+}1,N]$. By 
definition we have $k \leq N_\delta$, where $N_\delta$ is defined by 
\eqref{Nepsbdd} with $\epsilon = \delta$. Using the representation 
\eqref{Gallun} and the fact that $\chi \le 2$, we conclude that
\[
 |u_N|_B \,\le\, 2^k |u_{N-k}|_B \,\le\, 2^{N_\delta}\delta \,\le\,
 \delta \,2^{\frac{2}{h(\delta)}+1}\,,
\]
which proves \eqref{Mdelest}. Note that the right-hand side of 
\eqref{Mdelest} converges to zero as $\delta \to 0$ provided
\[
  \frac{2\ln(2)}{h(\delta)} + \ln(\delta) \,\to\, -\infty\,,
  \quad \hbox{as}\quad \delta \to 0\,,
\]
which is the case, for instance, if $h(s) \ge C|\ln s|^{-1}$ for 
some $C > 2\ln(2)$. 
\end{proof}

\begin{remark}\label{abmore}
More generally, if ${\rm supp}(\chi) \subset [-a,a]$ for some
$a > 0$ and $0 \le \chi \le b$ for some $b > 1$, the bound 
\eqref{Mdelest} becomes
\[
  \max_{n \in \N}|u_n|_B \,\le\, \delta\,b^{\frac{2a}{h(\delta)}+1}\,,
  \qquad \hbox{provided}\quad |u_0|_B \le \delta\,. 
\]
\end{remark}

In Example~\ref{Gallex}, the map $\cF$ defined by \eqref{def:Gal} is not Fr\'echet 
differentiable at the origin. Indeed, there holds
\[
   \lim_{\lambda \to 0+} \lambda^{-1}\cF(\lambda u) \,=\,
   Lu \,=\, \chi u\,, \qquad \hbox{for any }u \in B\,,
\]
but the convergence is not uniform in $u$ over the unit sphere $S = \{u
\in B \,|\, |u|_B = 1\}$. Uniform convergence holds on any subset of
$S$ that is bounded in the Sobolev space $H^s(\R)$ for some $s > 0$,
but such subsets are not invariant under the evolution defined by
$\cF$. Note however that, in view of \eqref{Gallun}, the large Fourier
modes of the initial data $u_0$ are only moderately amplified under
evolution, because the function $\chi$ is assumed to be smooth. If on
the contrary we take $\chi = 2 {\bf 1}_{[-1,1]}$, then
Proposition~\ref{Gallprop} remains valid, but the spectrum of $L$ now
consists of two eigenvalues $0$ and $2$, so that $L$ has a spectral
gap, and the generalized Lyapunov instability theorem mentioned in the
introduction shows that nonlinear stabilization is impossible if the
map is Fr\'echet differentiable. In that case, large Fourier modes are
immediately created because the Fourier transform of $\chi$ decays
slowly at infinity.

\subsection{A variant with additional contraction of the support}
\label{subsec32}
The following variant of Example~\ref{Gallex} exhibits an even more
efficient mechanism of nonlinear stabilization. Consider the Banach
space $B = C^0_0([-1,0];\R)$ of all continuous functions $u : [-1,0] 
\to \R$ satisfying $u(-1) = 0$, equipped with the norm $|u|_\infty = 
\max_{-1 \le x \le 0} |u(x)|$. Let $E$ be the extension operator 
defined for any $u \in B$ by $(E u)(x) = u(x)$ if $x \in [-1,0]$, 
and $(E u)(x) = 0$ if $x \in \R\setminus [-1,0]$. We define a map 
$\cF : B \to B$ by  
\begin{equation}\label{nF}
  (\cF(u))(x)  \,=\, 2 (E u)(2x-|u|_{\infty}^2)\,, \qquad 
  x \in [-1,0]\,,
\end{equation}
and we introduce the associated linearized operator
\[
  (Lu)(x) \,=\, 2(E u)(2x)\,, \qquad x \in [-1,0]\,.
\]
As before, it is easy to verify that $L$ is the G\^ateaux derivative of
$\cF$ at $u = 0$, but that $\cF$ is not Fr\'echet differentiable at
the origin. Moreover, we obviously have $|Lu|_\infty = 2|u|_\infty$
for any $u \in B$, hence $|L^n u|_\infty = 2^n|u|_\infty$ for all 
$n \ge 1$. Thus $r(L) = 2$, indicating linear exponential 
instability.

\begin{proposition}\label{one} Given $u_0 \in B$ the solution
of \eqref{de} satisfies, for any $\a \in [0,1]$, 
\begin{equation}\label{nlstab_infty}
  |u_n|_\infty \,\le\, 2^{3(1-\a)/2}\,2^{n(3\a - 1)/2} \,|u_0|_\infty^\a\,, \qquad 
  n \,\ge\, 1\,.
\end{equation}
Taking $0 < \a < 1/3$ proves nonlinear asymptotic stability.  
\end{proposition}

\begin{proof}
We have, evidently, 
\begin{equation}\label{ii}
  |u_n|_\infty \,\le\, 2^n | u_0|_\infty\,, \qquad 
  \hbox{and}\quad {\rm supp}\, u_n \,\subset\, [-2^{-n},0]\,.
\end{equation}
Let $I = \{n \in \N\,|\, |u_n|_\infty \le 2^{-n/2}\}$. If $n \notin
I$, then $|u_n|_\infty^2 > 2^{-n}$, and using the information on the
support in \eqref{ii} together with definition \eqref{nF} we deduce
that $u_{n+1} \equiv 0$, so that $n + 1 \in I$. As a consequence, 
if $n \notin I$ and $n \ge 1$, we must have $n{-}1 \in I$, hence 
$|u_n|_\infty \le 2 |u_{n-1}|_\infty \le 2\cdot2^{-(n-1)/2}$. Thus we have 
shown that
\[
  |u_n|_\infty \,\le\, 2^{-(n-3)/2}\,, \qquad \hbox{for all }
  n \ge 1\,,
\]
and interpolating with the bound in \eqref{ii} we easily obtain
estimate \eqref{nlstab_infty}.
\end{proof}

As in Example~\ref{Gallex} above, nonlinear stabilization occurs here
because the map $\cF$ involves a translation of the argument $u$,
whose support (under the repeated action of $\cF$) eventually leaves
the interval $[-1,0]$ where linear instability is at play.  The
novelty is that $\cF$ also shrinks the support of $u$ by a factor of
$2$, so that the support of $u_n$ for large $n$ is contained in a very
small one-sided neighborhood of the origin, and can thus easily be
pushed away from the interval $[-1,0]$.  This explains why
stabilization can be realized here with a very small spatial shift
$|u|_\infty^2$, whereas a larger translate $h(|u|_B)$ was necessary in
Example~\ref{Gallex}.  However, the contraction of the support has
also the effect of creating large Fourier modes in the solution $u_n$,
so that the modified map \eqref{nF} is certainly further away from
being Fr\'echet differentiable than the original map \eqref{def:Gal}.
As was already mentioned, failure of Fr\'echet differentiability at
the origin implies failure of the remainder bound \eqref{remest2},
hence Proposition~\ref{one} does not contradict the Rutman-Daleckii
theorem mentioned in the introduction.

\subsection{A hyperbolic partial differential equation.}
\label{subsec33}
Interestingly enough, the phenomenon illustrated in the previous 
example can occur if $\cF$ is the time-one map associated 
with an autonomous partial differential equation. To see this, 
consider the quasilinear hyperbolic equation
\begin{equation}\label{localpde}
  u_t +\big((-x+u^2)u\big)_x \,=0\,, \qquad x \in [-1,0]\,,
  \qquad u(-1,t) \,=\, 0\,.
\end{equation}
We assume that the initial data $u_0$ belong to the 
convex cone
\[
  B_+ \,=\, \{u \in B \,|\, u \hbox{ is nondecreasing}\}\,,
\]
where $B = C^0_0([-1,0];\R)$ is the same function space
as in Section~\ref{subsec32}. For such data Eq.~\eqref{localpde} 
has a unique global solution for positive times, which can be 
constructed by the method of characteristics and satisfies 
$u(\cdot,t) \in B_+$ for all $t \ge 0$. 

Indeed, the characteristic 
curve $X(t)$ originating from point $x_0 \in [-1,0]$ satisfies the 
differential equation $X'(t) = -X(t) + 3 u(X(t),t)^2$ with initial 
condition $X(0) = x_0$. Since
\[
  \frac{\dd}{\dd t}\,u(X(t),t) \,=\, u_t(X(t),t) + [-X(t) + 3 
  u(X(t),t)^2]u_x(X(t),t) \,=\, u(X(t),t)\,,
\]
we have $u(X(t),t) = e^t u_0(x_0)$, which in turn implies that
\begin{equation}\label{Xchar}
  X(t) \,=\, e^{-t}\Bigl(x_0 + u_0(x_0)^2(e^{3t}-1)\Bigr)\,,
  \qquad t \,\ge\, 0\,.
\end{equation}
As $u_0 \in B_+$, for each $t \ge 0$ the right-hand side of
equation \eqref{Xchar} is a strictly increasing function of $x_0 \in
[-1,0]$, which maps $[-1,0]$ onto $[-e^{-t}, a_0]$, with $a_0 = e^{-t}
u_0(0)^2 (e^{3 t} - 1) \ge 0$. In particular, for each $t \ge 0$
and each $x \in [-e^{-t},0]$ there exists a unique $x_0 \in [-1,0]$
such that $X(t) = x$.  Denoting $x_0 = X_0(x,t)$, we obtain the
representation formula
\begin{equation}\label{urep}
  u(x,t) \,=\, \begin{cases} e^t u_0(X_0(x,t)) & \!\!\hbox{if } 
  \,x \in [-e^{-t},0]\,,\\ 0 & \!\!\hbox{if } \,x \in [-1,-e^{-t}]\,,
\end{cases} 
\end{equation}
which gives a global solution to \eqref{localpde} for 
positive times, such that $u(\cdot,t) \in B_+$ for all $t \geq 0$.

Moreover, the time-$t$ map defined by Eq.~\eqref{urep} is G\^ateaux
differentiable at the origin in the cone $B_+$, with derivative given
by the time-$t$ map of the linearized equation $\tilde u_t - (x\tilde
u)_x = 0$. Indeed, given $u_0$ in $B_+$ and $\lambda > 0$, consider 
the function $u_\lambda$ defined in such a way that $\lambda u_\lambda(x,t)$ 
is the (unique) solution of \eqref{localpde} with initial data 
$\lambda u_0 \in B_+$. For $t > 0$ and $x \in [-e^{-t}, 0]$, we 
know from \eqref{Xchar}, \eqref{urep} that
\begin{equation} \label{urepl}  
  u_\lambda(x,t) \,=\, e^t u_0(X_{0,\lambda}(x,t))\,, \qquad \mbox{with} \quad 
  e^t x \,=\, X_{0,\lambda}(t,x) + \lambda^2 u_0(X_{0,\lambda}(t,x))^2 
  (e^{3 t} - 1)\,.
\end{equation}
Since $e^t x - \lambda^2 |u_0|_\infty (e^{3 t}-1) \le X_{0,\lambda}(x,t) 
\le e^t x$, we deduce from \eqref{urep} and \eqref{urepl} that, 
for any fixed $t > 0$, the function $u_\lambda(\cdot,t)$ converges 
uniformly on $[-1,0]$ to $\tilde u(\cdot,t)$ as $\lambda \to 0$, 
where
\begin{equation}\label{tildurep}
  \tilde u(x,t) \,=\, \begin{cases} e^t u_0(e^t x) & \!\!\hbox{if } 
  \,x \in [-e^{-t},0] \\ 0 & \!\!\hbox{if } \,x \in [-1,-e^{-t}]
\end{cases} 
\end{equation} 
is the unique solution in $B_+$ to the linearized equation
$\tilde u_t - (x\tilde u)_x = 0$ with initial data $u_0$. 

We observe that $|\tilde u(\cdot,t)|_\infty = e^t |u_0|_\infty$ for
any $t \ge 0$, so that the linearized evolution is exponentially
unstable, but the following result shows that the nonlinear evolution
is asymptotically stable.

\begin{proposition}\label{tthree}
For all initial data $u_0 \in B_+$ the solution of 
\eqref{localpde} given by \eqref{urep} satisfies, for any 
$\a \in [0,1]$, 
\begin{equation}\label{ustab}
  |u(\cdot,t)|_\infty \,\le\, C^{1-\a}  \,e^{(3\a-1)t/2} 
  \,|u_0|_\infty^\a\,, \qquad t \ge 1\,,
\end{equation}
where $C = (1-e^{-3})^{-1/2}$. Taking $0 < \a < 1/3$ proves 
nonlinear asymptotic stability.  
\end{proposition}

\begin{proof}
Given $t > 0$ and $x \in [-e^{-t},0]$, we denote $x_0 = X_0(x,t) 
\in [-1,0]$. We know from \eqref{Xchar} that $x_0 + u_0(x_0)^2
(e^{3t}-1) = e^t x \le 0$, hence we deduce from \eqref{urep}
that
\[
   0 \,\le\, u(x,t) \,=\, e^t u_0(x_0) \,\le\, 
   \frac{e^t |x_0|^{1/2}}{(e^{3t}-1)^{1/2}}\,.  
\]
Assuming $t \ge 1$ this gives the estimate $|u(\cdot,t)|_\infty \le C 
e^{-t/2}$, and interpolating with the trivial bound $|u(\cdot,t)|_\infty 
\le e^t |u_0|_\infty$ we obtain the desired result. 
\end{proof}

\begin{remark}\label{kruzhkov}
The restriction of the analysis to the convex cone $B_+$ represents a considerable
simplification, because no shocks can develop and the solution can be
constructed using the method of characteristics. For general initial
data $u_0 \in B$ with no monotonicity assumption, global existence of
a unique entropic weak solution to the scalar conservation law
\eqref{localpde} can be established following the approach of Kruzkhov
\cite{Kr}. However this solution may now have discontinuities, hence it
is necessary to work in a larger function space, such as
$L^\infty([-1,0])$.  In this more general setting, the time-$t$ map
associated with \eqref{localpde} is not differentiable at the origin,
even in the sense of G\^ateaux, so that the definition of the
linearized system becomes a more delicate question. These
technicalities are not related in any way to the nonlinear
stabilization effect that we discuss here, hence we prefer avoiding
them by working in the cone $B_+$, which however is not a Banach
space.
\end{remark}

\subsection{A finite-dimensional example}
\label{subsec34}
To conclude this section, we exhibit a  two-dimensional 
G\^ateaux differentiable map for which nonlinear stabilization 
occurs via a similar mechanism as in the previous examples. 
In $B = \R^2$ we consider the map $\cF$ defined by
\begin{equation}\label{F2dim}  
   \cF(v,w) \,=\, (2v \mathbf{1}_{D^c}(v,w) \,,\, w/2+ v^2/4)\,,
   \qquad (v,w) \in B\,,
\end{equation}
where $D = \{(v,w) \,|\, 0 < |w| < v^2\} \subset \R^2$ and
$\mathbf{1}_{D^c}$ denotes the indicator function of $D^c =
\R^2\setminus D$. Note that $D^c$ contains the axis $w = 0$, and that
$\cF(v,0) = (2v,v^2/4)$. The map $\cF$ is not continuous, but it is
continuous at the origin and G\^ateaux differentiable there with
derivative $L(v,w) = (2v,w/2)$, which is both linear and exponentially
unstable. However, it is easy to show that the origin is
asymptotically stable for the dynamics associated with the full map
\eqref{F2dim}.

\begin{proposition}\label{prop2dim}
For all initial data $(v_0,w_0) \in \R^2$ the solution of 
\eqref{de} with $\cF$ given by \eqref{F2dim} satisfies, 
for all $n \in \N$,
\begin{equation}\label{2dimest}
  |v_n|^2 + |w_n| \,\le\, 4\,\Bigl(\frac34\Bigr)^{n-1}\bigl(v_0^2+|w_0|\bigr)\,.
\end{equation}
\end{proposition}

\begin{proof}
If $v_0 = 0$, then $v_n = 0$ and $w_n = 2^{-n}w_0$ for all $n \in \N$, 
hence \eqref{2dimest} obviously holds. Moreover, if $(v_0,w_0) \in D$, 
then $v_1 = 0$ and $w_1 = w_0/2 + v_0^2/4$, so that \eqref{2dimest} 
holds for $n \le 1$, and the subsequent values of $n$ follow as 
in the previous case. Similarly, if $w_0 = 0$ and $v_0 \neq 0$, then
$(v_1,w_1) = (2v_0,v_0^2/4) \in D$, and proceeding as above we 
conclude that \eqref{2dimest} holds for all $n \in \N$. So it 
remains to consider the case where $|w_0| \ge v_0^2 > 0$. 
If $|w_n| \ge v_n^2$ for all $n \in \N$, we have $|w_{n+1}| 
\le |w_n|/2 +  v_n^2/4 \le 3|w_n|/4$ for all $n$, hence
\begin{equation}\label{2dimbis}
   |w_n| \,\le\, \Bigl(\frac34\Bigr)^n |w_0|\,, \qquad
   v_n^2 \,\le\, |w_n| \,\le\, \Bigl(\frac34\Bigr)^n |w_0|\,, 
\end{equation}
and \eqref{2dimest} follows. In the converse case, let $N \ge 1$ be
the smallest integer for which $|w_N| < v_N^2$. Then the first
inequality in \eqref{2dimbis} holds for all $n \le N$, and the second
one for $n \le N-1$. Since $v_N^2 \le 4v_{N-1}^2 \le 4\cdot(3/4)^{N-1}|w_0|$, 
we deduce that \eqref{2dimest} holds for all $n \le N$. We also know 
that $|w_N| \ge 4^{-N}|w_0| > 0$, hence $(v_N,w_N) \in D$, and this 
implies that $v_n = 0$ for all $n \ge N+1$. As  
\[
  |w_{N+1}| \,\le\, \frac{|w_N|}{2} + \frac{v_N^2}{4} \,<\, 
  \frac{3v_N^2}{4} \,\le\, 4\Bigl(\frac34\Bigr)^N |w_0|\,,
\]
we conclude that estimate \eqref{2dimest} holds for $n = N+1$, 
hence also for all subsequent $n$. The proof is thus complete. 
\end{proof}

\begin{remark}\label{Liprem}
The map $\cF$ defined in \eqref{F2dim} is not continuous, but we
believe that the same stabilization phenomenon can occur for a map
that is H\"older continuous and smooth outside the origin. Lipschitz
regularity cannot be achieved, because in a finite-dimensional space a
Lipschitz map $\cF$ that is G\^ateaux differentiable at the origin
with a linear derivative $L=d\cF|_{u=0}$ is automatically Fr\'echet
differentiable \cite[Appendix A]{AH}, in which case nonlinear
stabilization is precluded by the Lyapunov instability theorem.
\end{remark}


\section{The case of normal linear tangent maps}
\label{sec:normal}

In this final section, we assume that $B$ is a Hilbert space, and that
$L$ is a bounded linear operator in $B$ which is {\em normal}, in the
sense that $L L^* = L^* L$.  We consider a map $\cF : B \to B$ with
$\cF(0) = 0$ satisfying \eqref{remest3}, namely
\begin{equation}\label{remest4}
  |\cF(u) - Lu|_B \,\le\, \alpha(|u|_B)|u|_B\,, 
  \qquad\hbox{whenever}\quad |u|_B\le a\,,
\end{equation}
where $a > 0$ and $\alpha : [0,a] \to [0,+\infty)$ is a nondecreasing
function such that $\alpha(s) \to 0$ as $s \to 0$. This of course implies 
that $\cF$ is Fr\'echet differentiable at the origin with $L=d\cF|_{u=0}$. 
In this situation, the Rutman-Daleckii theorem quoted in the introduction 
can be sharpened as follows.

\begin{proposition}\label{nprop} If the function $\alpha$ satisfies the 
integrability condition \eqref{intalpha}, then the linear instability
assumption $r(L) > 1$ implies implies nonlinear exponential 
instability of system \eqref{de}, at the linear rate $\rho  = r(L)$. 
\end{proposition}

\begin{proof} We adapt the proof given in \cite[Theorem 5.1.5]{He}. 
Given $u_0 \in B$, let $(u_n)$ be the sequence defined by \eqref{de}. 
By straightforward induction, there holds
\begin{equation} \label{de:eq}
  u_n \,=\, L^n u_0 ~+ \sum_{0 \le k \le n-1} L^{n-k-1} 
  \big(\cF(u_k)- Lu_k\big)\,, \qquad n \,\ge\, 1\,,
\end{equation}
where $L^0$ is the identity map. Since the operator $L$ is normal, 
the norm $|L|_{B \to B}$ is equal to the spectral radius $r(L)$, 
see \cite[Section V.2.1]{Ka1}, which gives the  bound
\begin{equation} \label{normal:eq}
  |L u|_B \,\le\, r(L) |u|_B\,, \qquad \hbox{for all }u \in B\,.
\end{equation}
From \eqref{de:eq}, \eqref{normal:eq}, and assumption
\eqref{remest4}, we deduce the upper bound
\begin{equation} \label{de:bd}
  |u_n|_B \,\le\, r(L)^n |u_0|_B ~+ \sum_{0 \le k \le n-1} r(L)^{n-k-1} 
  \a(|u_k|_B) |u_k|_B\,,
\end{equation}
which holds provided $|u_k|_B \le a$ for $0 \le k \le n-1$.
 
From now on, we assume that $r(L) > 1$ and that the function 
$\alpha$ satisfies the integrability assumption \eqref{intalpha}. 
We fix $\eta \in (0,a]$ small enough so that 
\begin{equation} \label{cond:eta}
 \frac{2}{r(L) \ln r(L)} \int_0^{\eta} \frac{\a(s)}{s} \dd s 
 \,\le\, \frac14\,.
\end{equation}
Given any sufficiently small $\delta > 0$, we denote by $N = 
N(\delta)$ the unique positive integer such that
\begin{equation} \label{def:N}
  2 r(L)^N \delta \,\le\, \eta \,<\, 2 r(L)^{N+1} \delta\,.
\end{equation} 

In a first step, given any initial data $u_0 \in B$ with 
$|u_0|_B \le \delta$, we prove by induction that 
\begin{equation} \label{ind:henry}
 |u_k|_B \,\le\, 2 r(L)^k |u_0|_B\,, \quad 0 \le k \le N\,. 
\end{equation}
Indeed the inequality in \eqref{ind:henry} obviously holds 
for $k = 0$. Given $1 \le n \le N$, assume that the inequality 
holds for $0 \le k \le n-1$. Then $|u_k|_B \le 2 r(L)^N \delta 
\le \eta$ for $k \le n-1$, and using in addition \eqref{de:bd} 
we deduce that
\begin{equation} \label{for:ind}
 |u_n|_B \,\le\, r(L)^{n} \Big(1 + 2 r(L)^{-1} \sum_{0 \le k \le n-1} 
 \a\big(2 r(L)^{k} |u_0|_B\big)\Big) |u_0|_B\,.
\end{equation}
As the function $\alpha$ is nondecreasing by assumption, 
we can bound
\begin{align}\nonumber
  \sum_{0 \le k \le n-1} \a\big(2 r(L)^{k} |u_0|_B\big) \,\le\, 
  \int_{0}^{n} \a(2 r(L)^x |u_0|_B)\dd x \,&=\, \frac{1}{\ln r(L)} 
  \int_{2 |u_0|_B}^{2 r(L)^{n} |u_0|_B} \frac{\a(s)}{s}\dd s \\ 
  \,&\le\, \frac{1}{\ln r(L)} \int_0^\eta \frac{\a(s)}{s}\dd s
  \,\le\, \frac{r(L)}{8}\,, \label{up:alpha} 
\end{align}
where we have used \eqref{cond:eta} and \eqref{def:N} in the last
inequalities. Combining \eqref{for:ind} and \eqref{up:alpha}, we
obtain $|u_n|_B \le (5/4) r(L)^n |u_0|_B \le 2 r(L)^n |u_0|_B$, which
completes the inductive proof of \eqref{ind:henry}. 

In a second step, we prove nonlinear exponential instability of system
\eqref{de} by considering specific initial data. As the spectrum of
$L$ is a compact subset of the complex plane, there exists $\lambda
\in \sigma(L)$ such that $|\lambda| = r(L)$.  Moreover $\lambda$ is an
approximate eigenvalue of $L$ in the sense that, for any $\nu > 0$,
there exists $v_{\nu} \in B$ such that
\begin{equation}\label{approx}
 |v_{\nu}|_B = 1\,, \qquad \hbox{and}\quad
 |(L -\lambda) v_{\nu}|_B \,\le\, \nu\,.
\end{equation}
Indeed, take any $\mu \in \C$ with $|\mu| > r(L)$ and $|\mu - \lambda|
< \nu/2$. The norm of the resolvent operator $(L-\mu)^{-1}$ is equal
to $\mathrm{dist}(\mu,\sigma(L))^{-1} \ge |\mu -\lambda|^{-1} >
2/\nu$, hence there exists $v_{\nu} \in B$ with $|v_{\nu}|_B = 1$ such
that $|(L -\mu) v_{\nu}|_B \,\le\, \nu/2$, and \eqref{approx}
follows. Using the factorization
\[
  L^n -\lambda^n \,=\, (L^{n-1} + L^{n-2}\lambda + \dots + L 
  \lambda^{n-2} + \lambda^{n-1})(L-\lambda)\,,
\]
together with the bound \eqref{normal:eq} and the fact that $|\lambda| = r(L)$, 
we easily obtain
\begin{equation} \label{approx:ind}
  |(L^n - \lambda^n) v_\nu|_B \,\le\, \nu n r(L)^{n-1}\,,
  \qquad \hbox{for all } n \,\ge\, 1\,.
\end{equation}
Now, given $\delta > 0$ arbitrarily small, we define $N \in \N$ 
as in \eqref{def:N} and take $v_\nu \in B$ such that 
\eqref{approx} holds with $\nu = r(L)/(4N)$. We choose as 
initial data $u_0 = \delta v_\nu$, so that $|u_0|_B = \delta$. 
In view of \eqref{approx:ind}, we have
$$ 
  | L^n u_0|_B \,\ge\, r(L)^n \Bigl(1 - \frac{\nu n}{r(L)}\Bigr)
  |u_0|_B \,\ge\, \frac34 r(L)^n |u_0|_B\,, \qquad 
  1 \le n \le N\,.
$$ 
Thus, using \eqref{de:eq} as well as the estimates \eqref{ind:henry} 
and \eqref{up:alpha}, we obtain the lower bound
\begin{equation} \label{low0}
 |u_n|_B \,\ge\, r(L)^n\Big(\frac34 - 2 r(L)^{-1} \sum_{0 \le k \le n-1} 
 \a\big(2 r(L)^k |u_0|_B\big)\Big)|u_0|_B \,\ge\, \frac12 
 r(L)^n |u_0|_B\,,
\end{equation}
for $1 \le n \le N$. In particular, in view of \eqref{def:N}, there
holds $|u_N|_B \ge \epsilon := \eta/(4r(L))$, and this shows that
system \eqref{de} is exponentially unstable at the linear rate $\rho
= r(L)$ in the sense of Definition~\ref{def:nins}.
\end{proof}

\begin{remark}\label{addrem}
We develop below a more geometric approach to nonlinear instability, 
based on the construction of invariant regions, which gives in particular 
an alternative proof of Proposition~\ref{nprop}. The above proof 
is more straightforward, and it is interesting to note that the 
normality of the linearized operator $L$ is only used to establish
estimate \eqref{normal:eq}, which is a first step in the derivation 
of the upper bound \eqref{de:bd}.
\end{remark}

The following example shows that Proposition~\ref{nprop} is sharp
in terms of preserving the linear rate.

\begin{proposition}\label{ex:scal} For the scalar map $\cF: \R \to \R$ 
defined by 
\begin{equation}\label{scal}
  \cF(u) \,=\, r(L) u - \a(|u|) u\,,
\end{equation}
where $\a: \R_+ \to \R_+$ is nondecreasing on the interval $[0,a]$ 
with $\a(s) \to 0$ as $s \to 0$, nonlinear exponential instability 
of system \eqref{de} at the linear rate $r(L) > 1$ implies 
the integrability condition \eqref{intalpha}. 
\end{proposition}

\begin{proof} 
Assume nonlinear exponential instability of system \eqref{de} at the
linear rate: there exists $\e > 0$ and $C > 0$ such that, for
arbitrarily small initial data $u_0 \in \R$, the solution of
\eqref{de} satisfies $|u_n| \ge C r(L)^n |u_0|$ as long as $|u_n| \le
\e$. Without loss of generality, we suppose that that $\e \le a$ and
that $\a(s) \le r(L)/2$ for all $s \in [0,\e]$. Let $(u_n)$ be an
unstable sequence as described above, and assume for definiteness that
$u_0 > 0$. Let $N$ be the largest nonnegative integer such that
\begin{equation}\label{N} 
  C r(L)^n u_0 \,\le\,  u_n \,\le\,  \e\,, 
  \qquad \hbox{for} \quad 0 \le n \le N\,.
\end{equation}
By definition of $\cF$ in \eqref{scal}, there holds
\begin{equation}\label{unrep}
  u_n \,=\, r(L)^n u_0 \prod_{0 \le k \le n-1} 
  \Bigl(1 - \frac{\a(u_k)}{r(L)}\Bigr)\,, \qquad n \in \N\,,
\end{equation}
so that $u_n \le r(L)^n u_0$ for $0 \le n \le N+1$. It follows that 
\[
  \ln u_n \,=\, \ln u_0 + n \ln r(L) ~+ \sum_{0 \le k \le n-1} 
  \ln\Bigl(1 - \frac{\a(u_k)}{r(L)}\Bigr)\,, \qquad 1 \le n \le N\,,
\]
and since $u_n \ge C r(L)^n u_0$ this implies
\begin{equation}\label{sumsum}
  \frac{1}{r(L)}\sum_{0 \le k \le n-1} \a(u_k) \,\le\, 
  -\sum_{0 \le k \le n-1} \ln\Bigl(1 - \frac{\a(u_k)}{r(L)}\Bigr) 
  \,\le\, - \ln C\,, \qquad 1 \le n \le N\,,
\end{equation}
where the first inequality results from the fact that $x \le 
-\ln (1 - x)$ for $x \in [0,1)$. Using the monotonicity 
assumption on $\a$ and the lower bound $u_k \ge C r(L)^k u_0$, 
we deduce from \eqref{sumsum} that
\[
  \int_{-1}^{N-1}\a\big(C r(L)^x u_0\big) \dd x ~\le\, 
  \sum_{0 \le k \le N - 1} \a\big( C r(L)^k u_0 \big) \,\le\, 
  - r(L) \ln C\,,
\]
hence also
\[
  \int_{C r(L)^{-1}u_0}^{C r(L)^{N-1} u_0} \frac{\a(s)}{s} \dd s \,\le\, 
  - r(L) \ln r(L) \ln C\,.
\]
Our choice of $N$ implies that $r(L)^{N-1} u_0 \ge \e/r(L)^2$, 
whereas the lower bound $C r(L)^{-1}u_0$ in the above integral 
can be taken arbitrarily small. We conclude that 
\[
  \int_0^{C\e/r(L)^2} \frac{\a(s)}{s} \dd s \,\le\,  - r(L) 
  \ln r(L) \ln C\,,
\]
which of course implies \eqref{intalpha}.
\end{proof}

\begin{remark}\label{slower}
If $\alpha(s)$ converges arbitrarily slowly to zero as $s \to 0$, 
it is clear from the representation \eqref{unrep} that the 
origin $u = 0$ is exponentially unstable at any rate $\rho < r(L)$, 
but Proposition~\ref{ex:scal} shows that one can take $\rho = r(L)$ 
only if the integrability condition \eqref{intalpha} is satisfied. 
In the particular case where $\alpha(s) = |\ln s|^{-\gamma}$ near
$s = 0$, we thus have exponential instability at the linear 
rate $r(L)$ if $\gamma > 1$, and exponential instability 
at any rate $\rho < r(L)$ if $0  < \gamma \le 1$. In the limiting 
case $\gamma = 1$, a more detailed analysis shows that solutions 
of \eqref{de} with small initial data $u_0 > 0$ satisfy a lower 
bound of the form 
\[
  \frac{u_n}{|\ln u_n|^\sigma} \,\ge\, r(L)^n \frac{u_0}{
  |\ln u_0|^\sigma}\,,
\]
for some $\sigma > 0$, as long as $u_n$ remains sufficiently small. 
\end{remark}

\bigskip
In the rest of this section, we establish a general nonlinear
instability result which encompasses the particular situation
considered in Proposition~\ref{nprop}. Let $B_1$, $B_2$ be 
two Banach spaces, with $B_1 \neq \{0\}$. We consider the 
following discrete dynamical system in the product 
space $B_1 \times B_2$\:
\begin{equation}\label{prodsys}
  \left\{\begin{array}{l}
  \hspace{-0.2mm}v_{n+1} \,=\, \cF_1(v_n,w_n) \,:=\,
  L_1 v_n + {\mathcal N}_1(v_n,w_n)\,, \\
  \hspace{-1mm}w_{n+1} \,=\, \cF_2(v_n,w_n) \,:=\, 
  L_2 w_n + {\mathcal N}_2(v_n,w_n)\,,
  \end{array}\right.
\end{equation}
where $L_1, L_2$ are bounded linear operators in $B_1, B_2$, 
respectively, satisfying the following estimates
\begin{equation}\label{L1L2bounds}
  |L_1 v| \,\ge\, \rho |v|\,, \quad 
  |L_2 w| \,\le\, \rho |w|\,, \quad 
  \hbox{for all } (v,w) \in B_1 \times B_2\,, \quad \hbox{ where } \rho > 1.
\end{equation}
We assume that the nonlinear maps ${\mathcal N}_1 : B_1 \times B_2 \to B_1$ 
and ${\mathcal N}_2 : B_1 \times B_2 \to B_2$ vanish at the origin, and 
for any $r > 0$ we denote
\begin{equation}\label{alphadef}
  \alpha(r) \,=\, \sup\left\{\frac{|{\mathcal N}_1(v,w)| + 
  |{\mathcal N}_2(v,w)|}{|v|} ~\Big|~ 0 < |v| \le r\,,~ |w| \le 
  |v| \right\}\,.
\end{equation}
The nondecreasing function $\alpha : \R_+ \to \R_+$ measures the size
of the nonlinearity ${\mathcal N} = ({\mathcal N}_1,{\mathcal N}_2)$
only inside the truncated cone $\{(v,w)\,|\, 0 < |v| \le r, \, |w| \le
|v|\}$. Therefore, assuming that $\alpha(r) \to 0$ as $r \to 0$ does
not necessarily imply that the map $\cF = (\cF_1,\cF_2)$ is Fr\'echet
differentiable at the origin, with derivative $L = (L_1,L_2)$. 
Nevertheless, the following result shows that, if $\alpha(s)/s$ is
integrable at the origin, the linear exponential instability condition
\eqref{L1L2bounds} implies nonlinear instability of the origin for the
full system \eqref{prodsys}.

\begin{proposition}\label{prodprop}
Assume that \eqref{L1L2bounds} holds for some $\rho > 1$
and that the function $\alpha$ defined in \eqref{alphadef}
satisfies the integrability condition \eqref{intalpha} 
for some $a > 0$. Then the origin $(0,0)$ is an unstable 
equilibrium of \eqref{prodsys}. 
\end{proposition}

\begin{proof}
The idea is to construct an invariant region $D$ in a careful
way, depending on the function $\alpha$ which measures the size of the
nonlinearity. Let $r_0 > 0$ and $\beta : [0,r_0] \to [0,\infty)$ be a 
nondecreasing continuous function satisfying $\beta(r) \le r$ for 
$0 \le r \le r_0$. We denote
\begin{equation}\label{def:regD}
  D \,=\, \Bigl\{(v,w) \in B_1 \times B_2 \,\Big|\, 0 < |v| 
  \le r_0\,, ~|w| \le \beta(|v|)\Bigr\}\,.
\end{equation}
If $(v_n,w_n) \in D$, then
\begin{align} \label{low:v}
  |v_{n+1}| \,&\ge\,  |L_1v_n| - |{\mathcal N}_1(v_n,w_n)| \,\ge\, (\rho - 
  \alpha(|v_n|))|v_n|~, \\ \label{up:w}
  |w_{n+1}| \,&\le\, |L_2w_n| + |{\mathcal N}_2(v_n,w_n)| \,\le\, \rho \beta(|v_n|)
  + |v_n| \alpha(|v_n|)~.
\end{align}
To ensure that $|w_{n+1}| \le \beta(|v_{n+1}|)$ if $|v_{n+1}| \le r_0$, 
we impose the functional inequality
\begin{equation}\label{hineq}
  \rho \beta(r) + r\alpha(r) \,\le\, \beta(\rho r - r \alpha(r))~, \qquad 
  0 \le r \le r_0~.
\end{equation}
It is not clear a priori that \eqref{hineq} has any solution $\beta$
that is continuous and nondecreasing, but we shall see that the 
integrability condition \eqref{intalpha} is a necessary and 
sufficient condition for the solvability of \eqref{hineq}. 
For the moment we just observe that, since $\alpha$ a nondecreasing 
function, the hypothesis \eqref{intalpha} implies that $\alpha(r) \to 0$ 
as $r \to 0$. Thus, taking $r_0$ small enough, we can suppose that 
$\rho - \alpha(r_0) > 1$. 

We now look for a solution of \eqref{hineq} in the form
\[
  \beta(r) \,=\, Cr \int_0^r \frac{\alpha(s)}{s}\dd s~, \qquad 
  0 \le r \le r_0~,
\]
where $C > (\rho \ln(\rho))^{-1}$. For $r \le r_0$ we then have
\begin{align*}
  \beta(\rho r - r \alpha(r)) - \rho \beta(r) \,&=\, 
   C(\rho r - r \alpha(r))\int_0^{\rho r - r \alpha(r)} \frac{\alpha(s)}{s}
  \dd s - C\rho r  \int_0^r \frac{\alpha(s)}{s}\dd s \\
  \,&=\, C\rho r \int_r^{\rho r - r \alpha(r)} \frac{\alpha(s)}{s}\dd s
   -  Cr \alpha(r)\int_0^{\rho r - r \alpha(r)} \frac{\alpha(s)}{s}\dd s \\
  \,&\ge\, Cr\alpha(r)\Bigl(\rho\ln(\rho-\alpha(r)) - \int_0^{\rho r 
    - r \alpha(r)} \frac{\alpha(s)}{s}\dd s\Bigr)~,
\end{align*}
where in the last line the monotonicity of $\alpha$ was used. Since
the last integral converges to zero as $r \to 0$, and since $C \rho
\ln(\rho) > 1$, the right-hand side is larger than
$r\alpha(r)$ if $r$ is sufficiently small. Thus, if $r_0 > 0$ is small
enough, then $\beta(r) \le r$ for $0 \le r \le r_0$ and \eqref{hineq} is
satisfied. 

With $r_0$ and $\beta$ as above, consider the solution $(v_n, w_n)$ of
\eqref{prodsys} originating from arbitrarily small initial data $(v_0,
w_0) \in D$. As long as $|v_n| \le r_0$, the solution $(v_n,w_n)$
remains in the region $D$ defined by \eqref{def:regD}, as can be seen
using the bounds \eqref{low:v}, \eqref{up:w}, the functional
inequality \eqref{hineq}, and a straightforward induction. In
that region, the lower bound \eqref{low:v} implies that $|v_n| \ge
(\rho - \a(r_0))^n |v_0|$, which proves that the origin is
exponentially unstable at rate $\rho - \a(r_0) > 1$.
\end{proof}

\begin{remark}
Assume that we are given a solution $\beta$ of \eqref{hineq}. Since 
$\beta$ is nondecreasing, we have in particular $\rho \beta(r) + r\alpha(r) \le 
\beta(\rho r)$ for $0 \le r \le r_0$. Thus, for $\epsilon > 0$ small 
enough, 
\[
  \int_\epsilon^{r_0} \frac{\alpha(r)}{r}\dd r \,\le\, 
  \int_\epsilon^{r_0} \frac{\beta(\rho r)}{r^2}\dd r - \rho\int_\epsilon^{r_0} 
  \frac{\beta(r)}{r^2}\dd r \,=\, \rho \int_{\rho\epsilon}^{\rho r_0} 
  \frac{\beta(r)}{r^2}\dd r - \rho \int_\epsilon^{r_0} \frac{\beta(r)}{r^2}\dd r~,
\]
hence
\[
  \int_\epsilon^{r_0} \frac{\alpha(r)}{r}\dd r + \rho \int_\epsilon^{\rho 
  \epsilon} \frac{\beta(r)}{r^2}\dd r \,\le\, \rho \int_{r_0}^{\rho r_0} 
  \frac{\beta(r)}{r^2}\dd r~. 
\]
Taking $\epsilon \to 0$, we obtain $\int_0^{r_0} \alpha(r)/r\dd r <
\infty$. The integrability condition \eqref{intalpha} is thus
necessary and sufficient for a solution of \eqref{hineq} to exist.
\end{remark}

As a conclusion, we briefly indicate how Proposition~\ref{prodprop}
can be used to establish the nonlinear instability result in
Proposition~\ref{nprop}. In the Hilbert space $B$, consider the map
$\cF(u) = Lu + {\mathcal N}(u)$ whose Fr\'echet derivative $L =
d\cF|_{u=0}$ is a normal operator with spectral radius $r(L) >
1$. Given $1 < \rho < r(L)$, let $P_1 : B \to B$ be the spectral
projection corresponding to the (nonempty) subset of $\sigma(L)$
contained in the annulus $\{\lambda \in \C \,|\, \rho \le |\lambda|
\le r(L)\}$, and let $P_2 = 1 - P_1$.  Even in the absence of spectral
gap, the (orthogonal) projections $P_1, P_2$ can be constructed using
the spectral theorem for bounded normal operators. Let $B_1 = P_1B$,
$B_2 = P_2B$ so that $B_1 \neq \{0\}$ and $B = B_1 \oplus B_2$. If we
denote by $L_1, L_2$ the restrictions of $L$ to the invariant
subspaces $B_1, B_2$, respectively, then estimates \eqref{L1L2bounds}
hold by construction because $L_1, L_2$ are normal operators. Finally,
denoting ${\mathcal N}_1 = P_1 {\mathcal N}$, ${\mathcal N}_2 = P_2
{\mathcal N}$, and writing $u_n = v_n + w_n$ with $v_n = P_1 u_n$ and
$w_n = P_2 u_n$, we see that system \eqref{de} takes the form
\eqref{prodsys}, and assumption \eqref{remest4} on the nonlinearity
${\mathcal N}$ shows that the right-hand side of \eqref{alphadef} is
bounded from above by a multiple of $\alpha(r)$, where $\alpha$ is as
in \eqref{remest4}. Thus, under the integrability condition
\eqref{intalpha}, Proposition~\ref{prodprop} implies that the origin
$u = 0$ is nonlinearly exponentially unstable for system \eqref{de},
at any rate $\rho < r(L)$.
 
\begin{remark}
Proposition~\ref{prodprop} deals in principle with a more general
situation than Proposition~\ref{nprop}, but except in the normal case
considered above it is not clear under which assumptions a linear
operator $L$ without spectral gap can be decomposed into a strongly
unstable part $L_1$ and a weakly unstable part $L_2$ satisfying
estimates of the form \eqref{L1L2bounds}.
\end{remark} 

\section{Appendix : Stabilization in a two-dimensional system} 
\label{sec:app}

We briefly analyze the two-dimensional dynamical system 
associated with the map \eqref{2Dsys}, namely
\begin{equation}\label{2Dsys2}
  v_{n+1} \,=\, v_n + w_n - v_n^3\,, \qquad 
  w_{n+1} \,=\, w_n - w_n^3\,.
\end{equation}
Elementary calculations show that the origin $(u,v) = (0,0)$
in \eqref{2Dsys2} is asymptotically stable. 

\begin{proposition}\label{2Dprop}
If $|v_0| \le 1/2$ and $|w_0| \le 1/8$, the solution of \eqref{2Dsys2} satisfies
\begin{equation}\label{2Dbound}
  |v_n| \,\le\, \max(|v_0|,|w_0|^{1/3})\,,\qquad
  |w_n| \,\le\, |w_0|\,, \qquad \hbox{for all }n \in \N\,.
\end{equation}
In addition $|v_n| + |w_n| \to 0$ as $n \to \infty$. 
\end{proposition}

\begin{proof} 
We first assume that $0 \le v_0 \le 1/2$ and $0 \le w_0 \le 1/8$. 
Since $w_{n+1} = w_n(1-w_n^2)$, it is clear that the sequence $(w_n)$ 
is nonincreasing and converges to zero as $n \to \infty$, see 
also Remark~\ref{comparison} below. In particular we have 
$0 \le w_n \le w_0$ for all $n \in \N$. As for the first component, 
we show by induction that 
\begin{equation}\label{uncond}
  0 \,\le\, v_n \,\le\, \max(v_0,w_0^{1/3})\,, \qquad
  \hbox{for all }n \in \N\,.
\end{equation}
Indeed, assume that \eqref{uncond} holds for some $n \in \N$.  
If $w_n \le v_n^3$, then $v_{n+1} \le v_n \le \max(v_0,w_0^{1/3})$ 
and $v_{n+1} = v_n(1-v_n^2) + w_n \ge 0$. If $w_n \ge v_n^3$, then 
$v_{n+1} \ge v_n \ge 0$ and
\[
  v_{n+1} \,=\, v_n + w_n - v_n^3 ~\le\, \max_{0 \le u \le w_n^{1/3}}
  \bigl(u + w_n - u^3\bigr) \,=\, w_n^{1/3} \,\le\, w_0^{1/3}\,,
\] 
because the map $u \mapsto u - u^3$ is increasing on the interval
$[0,w_n^{1/3}] \subset [0,1/2]$. This shows that the bounds
\eqref{uncond}, which hold by assumption for $n = 0$, remain valid for
all $n \in \N$.  We also note that the inequality $w_n > v_n^3$ cannot
hold for all $n \in \N$, because in that case the sequence $(v_n)$
would be strictly increasing, and we know that the sequence $(w_n)$
decreases to zero. So there exists $n \in \N$ such that $\delta_n :=
v_n^3 - w_n \ge 0$. Using \eqref{2Dsys2} we deduce that
\begin{align*}
  \delta_{n+1} \,&=\, v_{n+1}^3 - w_{n+1} \,=\, (v_n - \delta_n)^3 - w_n + w_n^3 \\
  \,&=\, \delta_n(1-3v_n^2+3v_n\delta_n-\delta_n^2) + w_n^3 \,\ge\,
  \delta_n(1-3v_n^2-\delta_n^2) \,\ge\, 0\,,
\end{align*}
because $3v_n^2 + \delta_n^2 \le 3v_n^2 + v_n^6 \le 3/4 + 1/64 < 1$
(this follows from \eqref{uncond} and from our assumptions on the
initial data).  Thus we necessarily have $\delta_n \ge 0$ for all
sufficiently large $n \in \N$, which implies that the sequence $(v_n)$
is eventually nonincreasing, hence converges to some limit $\bar u \ge
0$ as $n \to \infty$. As $w_n \to 0$, it follows from \eqref{2Dsys2} 
that $\bar u = \bar u - \bar u^3$, hence $\bar u = 0$. This concludes 
the proof in the case where $v_0 \ge 0$ and $w_0 \ge 0$.

We next assume that $-1/2 \le v_0 \le 0$ and $0 \le w_0 \le 1/8$. 
In that case, as long as $v_n \le 0$, the sequence $(v_n)$ is 
nondecreasing because $w_n - v_n^3 \ge 0$. So either $v_n \le 0$
for all $n \in \N$, in which case $v_n \to 0$ as $n \to \infty$, 
or there exists a first integer $n \in \N$ such that $v_n \le 0$ and
$v_{n+1} > 0$. In that case we have $v_{n+1} = v_n(1-v_n^2) + w_n \le 
w_n$, hence $v_{n+1} \in [0,1/2]$, $w_{n+1} \in [0,1/8]$, and 
we are back to the situation studied above where both components 
are nonnegative. We deduce that \eqref{2Dbound} holds in all 
cases and that $|v_n| + |w_n| \to 0$ as $n \to \infty$. 

Finally, the same conclusions hold if $w_0 \le 0$, because system
\eqref{2Dsys2} is clearly invariant under the transformation 
$(v_n,w_n) \mapsto (-v_n,-w_n)$. This concludes the proof. 
\end{proof}

\begin{remark}\label{comparison}
In fact, a standard comparison argument with the continuous time 
ODE $w' = -w^3$ shows that, if $3w_0^2 \le 1$, the solution
$w_n$ of \eqref{2Dsys2} satisfies
\[
  |w_n| \,\le\, \frac{|w_0|}{\sqrt{1+2|w_0|^2n}}\,, \qquad 
  \hbox{for all }n \in \N\,.
\]
In particular $|w_n| = \cO(n^{-1/2})$ as $n \to \infty$.
\end{remark}

\end{document}